\documentclass[11pt]{article} 
\usepackage{amsmath, amssymb, latexsym, tabularx, url}
\newtheorem{defin}{}
\newtheorem{saetze}[defin]{}
\newtheorem{conjec}[defin]{}
\newtheorem{lemmas}[defin]{}
\newtheorem{folger}[defin]{}
\newtheorem{bemerk}[defin]{}

\newenvironment{theorem}  {\begin{saetze}\it {\bf Theorem:}}{\end{saetze}}
\newenvironment{conjecture}{\begin{conjec}\it {\bf Conjecture:}}{\end{conjec}}
\newenvironment{lemma}    {\begin{lemmas}\it {\bf Lemma:}}{\end{lemmas}}
\newenvironment{corollary}{\begin{folger}\it {\bf Corollary:}}{\end{folger}}
\newenvironment{remark}   {\begin{bemerk}\rm {\it Remark:}}{\end{bemerk}}
\newenvironment{proof}    {{\it Proof}:}{{\hfill \fillbox \bigskip}}

\newcommand{\fillbox}{\mbox{$\bullet$}}
\newcommand{\ra}{\rightarrow}

\newcommand{\ms}{\mapsto}
\newcommand{\ol}{\overline}

\newcommand{\N}{\mathbb N}

\newcommand{\Z}{\mathbb Z}
\newcommand{\Q}{\mathbb Q}
\newcommand{\C}{\mathcal C}

\renewcommand{\phi}{\varphi}

\newenvironment{items}{\begin{list}{$\alph{item})$}
{\labelwidth18pt \leftmargin18pt \topsep3pt \itemsep1pt \parsep0pt}}
{\end{list}}

\newcommand{\bulit}{\item[$\bullet$]}

\begin{document}

\title{Polynomials describing the multiplication in \\
       finitely generated torsion-free nilpotent groups}
\author{Alexander Cant and Bettina Eick}
\date{\today}
\maketitle

\begin{abstract}
A famous result of Hall asserts that the multiplication and exponentiation
in finitely generated torsion-free nilpotent groups can be described by
rational polynomials. We describe an algorithm to determine such
polynomials for all torsion-free nilpotent groups of given Hirsch
length. We apply this to determine the Hall polynomials for all such
groups of Hirsch length at most $7$.
\end{abstract}

\section{Introduction}

Let $G$ be a finitely generated torsion-free nilpotent group ($T$-group for
short). Then $G$ has a central series $G = G_1 \ge G_2 \ge \cdots \ge G_n 
\ge G_{n+1} = \{1\}$ with infinite cyclic factors. The number $n$ depends
only on $G$ and is called its \textit{Hirsch length}. Let $a_i G_{i+1}$ be a
generator for $G_i/G_{i+1}$ for $1 \leq i \leq n$. Then for every element
$g$ of $G$ there exists a unique $x = (x_1, \ldots, x_n) \in \Z^n$ with
$g = a_1^{x_1} \cdots a_n^{x_n}$. This allows to express
the multiplication and powering in $G$ by functions
$F_i \colon \Z^n \oplus \Z^n \ra \Z, (x,y) \ms F_i(x,y)$
and $K_i \colon \Z^n \oplus \Z \ra \Z, (x,z) \ms K_i(x,z)$ 
defined by
\[ (a_1^{x_1} \cdots a_n^{x_n}) 
   (a_1^{y_1} \cdots a_n^{y_n}) 
\ = \
   a_1^{F_1(x,y)} \cdots a_n^{F_n(x,y)}, \]
and
\[ (a_1^{x_1} \cdots a_n^{x_n})^z \ = \
   a_1^{K_1(x,z)} \cdots a_n^{K_n(x,z)}.\]
 
Hall \cite{Hal69} showed that the functions $F_1,\ldots, F_n$ and $K_1, 
\ldots, K_n$ can be described by rational polynomials in $x,y,z$; these 
polynomials are also called {\em Hall polynomials} nowadays. 

Hall polynomials have several interesting applications. The first and 
most obvious application is that they allow to multiply in a $T$-group
via evaluation of polynomials. This is often significantly faster than 
the multiplication via the usually used 'collection' algorithm. We 
discuss this in Section \ref{comp} below in more detail. 

Hall polynomials also have applications beyond the obvious multiplication. 
For example, they have been used in \cite{EEi16} to classify up to 
isomorphism the 
$T$-groups of Hirsch length at most $5$. Further, Nickel \cite{Nic06}
showed how the Hall polynomials of a $T$-group $G$ can be used to 
determine a faithful matrix representation for $G$. Recently, Gul
\& Weiss \cite{FWe17} investigated Nickel's method in detail with the
aim to find upper bounds on the dimension of the arising faithful matrix
representations.

It still remains difficult to determine the Hall polynomials for a given
$T$-group $G$. Leedham-Green \& Soicher \cite{LGS98} exhibit an algorithm 
called 'Deep Thought' to compute Hall polynomials for a $T$-group $G$ from 
a polycyclic presentation of $G$. Sims \cite[Sec 9.10]{Sim94} outlines
an alternative approach for the special case that $G$ is free nilpotent.
Further, Hall's original proof \cite{Hal69} is constructive and could be
translated into an algorithm.

The aim here is to describe an algorithm that, given $n \in \N$, determines 
Hall polynomials for {\em all} $T$-groups of Hirsch length $n$ {\em 
simultaneously}. We describe the $T$-groups of Hirsch length $n$ via 
polycyclic presentations containing parameters and the Hall polynomials 
are determined as polynomials incorporating these parameters. Our 
method extends and refines \cite{EEi16b}. It differs from this and from
all other previously known approaches in its reduction of the problem to 
solving certain polynomial recursions.

Our algorithm is implemented in the HallPoly package \cite{hallpoly} of 
the computer algebra system {\sf GAP} \cite{GAP}. We used this implementation
to determine the Hall polynomials of all $T$-groups of Hirsch length at 
most $7$. The resulting polynomials are exhibited in the HallPoly package.
Some information about them is included in Section \ref{impl} below.
\medskip

\textbf{Acknowledgements.} The authors thank Anne Fr\"uhbis-Kr\"uger for
helpful discussions on Groebner bases and the use of {\sc Singular}
\cite{DGPS}.

\section{Nilpotent presentations}
\label{pres}

Let $n \in \N$ and $t = (t_{i,j,k} \mid 1 \leq i < j < k \leq n)
\in \Z^{\binom{n}{3}}$. Then $t$ defines a group $G(t)$ via
\[ G(t) \ = \ \langle a_1, \ldots, a_n \mid 
    a_j a_i = a_i a_j a_{j+1}^{t_{i,j,j+1}} \cdots a_n^{t_{i,j,n}}
    \mbox{ for } 1 \leq i < j \leq n \rangle.\]
Let $G_i = \langle a_i, \ldots, a_n \rangle \leq G(t)$ for $1 \leq i 
\leq n+1$. The presentation of $G(t)$ implies that $G(t) = G_1 \geq \cdots 
\geq G_{n+1} = \{1\}$ is a central series of $G$ with cyclic factors 
$G_i/G_{i+1} = \langle a_i G_{i+1} \rangle$. 
Hence $G(t)$ is finitely generated nilpotent and has Hirsch length at most 
$n$. It follows readily that for each element $g \in G(t)$ there 
exists $x \in \Z^n$ so that 
\[ g = a_1^{x_1} \cdots a_n^{x_n}.\]

We say that $G(t)$ is {\em consistent} if $G(t)$ has Hirsch length precisely 
$n$. In this case $G(t)$ is a $T$-group and each quotient of the series 
$G(t) = G_1 \geq \cdots \geq G_{n+1} = \{1\}$ is infinite cyclic. Hence
for each element $g \in G(t)$ there exists a unique $x \in \Z^n$ with 
$g = a_1^{x_1} \cdots a_n^{x_n}$. We define
\[ \C_n = \{ t \in \Z^{\binom{n}{3}} \mid G(t) \mbox{ consistent} \}.\]

\begin{lemma}\label{HLGt}
For each $T$-group $G$ of Hirsch length $n$ there exists $t \in \C_n$ 
with $G \cong G(t)$.
\end{lemma}

\begin{proof}
Choose a central series $G = G_1 \geq \cdots \geq G_n \geq G_{n+1} 
= \{1\}$ with infinite cyclic quotients for $G$ and let $g_i G_{i+1}$ be
a generator for $G_i/G_{i+1}$. Then there exists a (unique) $t \in 
\Z^{\binom{n}{3}}$ so that $g_j g_i = g_i g_j g_{j+1}^{t_{i,j,j+1}} 
\cdots g_n^{t_{i,j,n}}$ holds for $1 \leq i < j \leq n$. Hence there
exists an epimorphism $G(t) \ra G$ that maps $a_i$ to $g_i$. The choice 
of $g_1, \ldots, g_n$ implies that every element in $G$ can be written 
uniquely as $g_1^{x_1} \cdots g_n^{x_n}$ for certain $x_1, \ldots, x_n 
\in \Z$. Thus $a_1^0 \cdots a_n^0$ is the only preimage of $1$ under the
epimorphism $G(t) \ra G$. This implies that the epimorphism is injective 
and hence an isomorphism.
\end{proof}

\section{Solving recursions}

The Bernoulli numbers $\{ B_l \mid l \in \N_0 \}$ are an infinite sequence
of rational numbers starting $B_0 = 1, B_1 = -\frac{1}{2}, B_2 = \frac{1}{6}, 
B_3 = 0, \ldots$. They satisfy the following equation for all $x,m \in \N$:
\[ \sum_{i=1}^{x-1} i^m 
   = \frac{1}{m+1} \sum_{k=0}^m B_k {\binom{m+1}{k}} x^{m-k+1}.\]

It is well-known that the Bernoulli numbers can be used to solve certain
polynomial recursions. We recall this in our special case as follows.
We include a short constructive proof for completeness.

\begin{lemma}
\label{bernoulli}
Let $R$ be a ring with $\Q \subseteq R$. Let $g(x) = c_0 + c_1 x + \cdots + 
c_l x^l \in R[x]$ and $f_0 \in R$. For $m > 0$ define 
\[ f_m = \sum_{k=0}^{l+1-m} 
               \frac{c_{m+k-1}}{m+k} B_k {\binom{m+k}{k}} \in R.\]
Then the polynomial $f(x) = f_0 + f_1 x + \cdots + f_{l+1} x^{l+1} \in R[x]$
satisfies $f(0) = f_0$ and $f(x+1) = f(x) + g(x)$ for all $x \in \N_0$.
\end{lemma}

\begin{proof}
By construction $f(0) = f_0$ holds and it remains to show the recursion.
Let $x \in \N$. Using $j = m+k-1$ we obtain the following:
\begin{eqnarray*}
f(x) 
  &=& f_0 + \sum_{m=1}^{l+1} f_m x^m 
  = f_0 + \sum_{m=1}^{l+1} \sum_{k=0}^{l+1-m} 
               \frac{c_{m+k-1}}{m+k} B_k {\binom{m+k}{k}} x^m \\
  &=& f_0 + \sum_{j=0}^l 
             \sum_{k=0}^j \frac{c_j}{j+1} B_k {\binom{j+1}{k}} x^{j-k+1}  \\
  &=& f_0 + \sum_{j=0}^l c_j 
           \Big(\frac{1}{j+1} \sum_{k=0}^j B_k {\binom{j+1}{k}} 
                x^{j-k+1}\Big) \\
  &=& f_0 + \sum_{j=0}^l c_j \sum_{i=1}^{x-1} i^j 
  = f_0 + \sum_{i=1}^{x-1} \sum_{j=0}^l c_j i^j 
  = f_0 + \sum_{i=1}^{x-1} g(i). 
\end{eqnarray*}
This implies that $f(x) = f(x-1) + g(x-1)$ for all $x \in \N$ and hence 
$f(x+1) = f(x) + g(x)$ for all $x \in \N_0$.
\end{proof}

\section{Hall polynomials}
\label{hall}

Let $n \in \N$ and let $T = (T_{i,j,k} \mid 1 \leq i < j < k \leq n)$ be
a sequence of ${\binom{n}{3}}$ indeterminates. This defines the parametrised 
presentation 
\[ G(T) \ = \ \langle a_1, \ldots, a_n \mid 
    a_j a_i = a_i a_j a_{j+1}^{T_{i,j,j+1}} \cdots a_n^{T_{i,j,n}}
    \mbox{ for } 1 \leq i < j \leq n \rangle.\]
By Lemma \ref{HLGt}, for each $T$-group $G$ of Hirsch length $n$ there exists
$t \in \C_n$ so that $G \cong G(t)$. Hence we can consider the parametrised
presentation $G(T)$ as a {\em generic} description for all $T$-groups of 
Hirsch length $n$. 

In this section we describe an algorithm that for $1 \leq i \leq n$ 
determines \textit{multiplication polynomials} $F_i(T;x,y)$ in the 
indeterminates
$x = (x_1, \ldots, x_n)$ and $y = (y_1, \ldots, y_n)$ and the parameters $T$, 
so that for all $t \in \C_n$ and each $x,y \in \Z^n$ the
equation
\[ (a_1^{x_1} \cdots a_n^{x_n})(a_1^{y_1} \cdots a_n^{y_n}) 
      = a_1^{F_1(t;x,y)} \cdots a_n^{F_n(t;x,y)} \mbox{ holds in } G(t). \]

Simultaneously, we describe an algorithm that for $1 \leq i \leq n$
determines \textit{powering polynomials} $K_i(T; x,z)$ for $1 \leq i 
\leq n$ in 
indeterminates $x = (x_1, \ldots, x_n)$ and $z$ and the
parameters $T$, so that for all $t \in \C_n$,
$x \in \Z^n$ and $z \in \Z$ the equation
\[ (a_1^{x_1} \cdots a_n^{x_n})^z 
      = a_1^{K_1(t;x,z)} \cdots a_n^{K_n(t;x,z)} \mbox{ holds in } G(t). \]

Additionally to the multiplication and powering polynomials we determine 
\textit{conjugation polynomials} $R_{i,j,k}(T;u,v)$ for $1 \leq i < j < 
k \leq n$
in the indeterminates $u,v$ and the parameters $T$, so that for all $t \in 
\C_n$ and every $u,v \in \Z$ the equation
\[ a_j^u a_i^v = a_i^v a_j^u 
   a_{j+1}^{R_{i,j,j+1}(t;u,v)} \cdots a_n^{R_{i,j,n}(t;u,v)}
   \mbox{ holds in } G(t).\]
The conjugation polynomials $R_{i,j,k}$ form a
key step towards constructing the multiplication polynomials $F_1,
\ldots, F_n$ and powering polynomials $K_1, \ldots, K_n$.

\subsection{An induction approach}
\label{induct}

Multiplication, powering and conjugation polynomials for the 
smallest cases $n \leq 2$ are trivial to determine, since $G(T)$ is 
free abelian in these cases. In the remainder of this section we 
consider $n > 2$ and we assume by induction that multiplication, powering
and conjugation polynomials for the case $n-1$ are known.
We define 
\begin{eqnarray*}
U(T) &=& \langle a_2, \ldots, a_n \rangle \leq G(T), \\
V(T) &=& \langle a_1, a_3, \ldots, a_n \rangle \leq G(T), \mbox{ and } \\
W(T) &=& G(T)/ \langle a_n \rangle.
\end{eqnarray*}
For $t \in \C_n$ the groups $U(t)$, $V(t)$ and $W(t)$ have Hirsch length 
$n-1$. It follows that 
\begin{eqnarray*}
U(T) &\cong& G(T_u), \mbox{ where } T_u = (T_{i,j,k} \mid i,j,k \neq 1), \\
V(T) &\cong& G(T_v), \mbox{ where } T_v = (T_{i,j,k} \mid i,j,k \neq 2), \\
W(T) &\cong& G(T_w), \mbox{ where } T_w = (T_{i,j,k} \mid i,j,k \neq n),
\end{eqnarray*}
in the sense that the isomorphism holds whenever $T$ is replaced by 
some $t \in \C_n$. Note that if $t \in \C_n$, then $t_u, t_v, t_w \in 
\C_{n-1}$ follows. Hence, by induction, we can assume throughout that 
the multiplication, powering and conjugation polynomials are given 
for $U(T)$, $V(T)$ and for $W(T)$.

\subsection{Conjugation polynomials}
\label{conj}

We consider the conjugation polynomials $R_{i,j,k}$
for $1 \leq i < j < k \leq n$. Using the observations of Section \ref{induct}, 
it is sufficient to determine the polynomial $R_{1,2,n}$, as 
all other polynomials are covered by the induction approach. The following
lemma is a first step towards calculating $R_{1,2,n}$.

\begin{lemma}
\label{step1}
$R_{1,2,n}(T;u,v) = K_{n-1}(T_u; r, u)$ for
$r = (1, R_{1,2,3}(T; 1, v), \ldots, R_{1,2,n}(T; 1, v))$. 
\end{lemma}

\begin{proof}
This can be proved by a direct calculation:
\begin{eqnarray*}
a_1^{-v} a_2^u a_1^v
 &=& a_2^u a_3^{R_{1,2,3}(T;u,v)} \cdots a_n^{R_{1,2,n}(T;u,v)} \\
&\mbox{ and }& \\
a_1^{-v} a_2^u a_1^v
&=& (a_1^{-v} a_2 a_1^v)^u \\
&=& (a_2 a_3^{R_{1,2,3}(T;1,v)} \cdots a_n^{R_{1,2,n}(T;1,v)})^u \\
&=& a_2^u a_3^{K_2(T_u; r, u)} \cdots a_n^{K_{n-1}(T_u; r, u)}. 
\end{eqnarray*}
Hence the desired result follows.
\end{proof}

Lemma \ref{step1} combined with the induction approach reduces the construction
of the polynomial $R_{1,2,n}(T; u, v)$ to the construction of the polynomial
$R_{1,2,n}(T; 1, v)$. The next lemma shows that this satisfies a polynomial
recursion. Its constructive proof translates readily to an algorithm to 
determine such a recursion for $R_{1,2,n}(T; 1, v)$. Write $\Q[T]$ for 
the polynomial ring over $\Q$ with indeterminates $T_{i,j,k}$ for $1 \le 
i < j < k \le n$.

\begin{lemma}
\label{step2}
There exists $g(v) \in \Q[T][v]$ so that
$R_{1,2,n}(T;1,v+1) = R_{1,2,n}(T;1,v) + g(v)$ for all $v \in \Z$.
\end{lemma}

\begin{proof}
We observe that
\begin{eqnarray*}
&& a_2 a_3^{R_{1,2,3}(T;1,v+1)} \cdots a_n^{R_{1,2,n}(T;1,v+1)} \\
&=& a_1^{-(v+1)} a_2 a_1^{v+1} \\
&=& a_1^{-1} (a_1^{-v} a_2 a_1^v) a_1 \\
&=& a_1^{-1} (a_2 a_3^{R_{1,2,3}(T;1,v)} \cdots a_n^{R_{1,2,n}(T;1,v)}) a_1 \\
&=& a_2^{a_1} (a_3^{a_1})^{R_{1,2,3}(T;1,v)} \cdots 
        (a_n^{a_1})^{R_{1,2,n}(T;1,v)}\\
&=& a_2 a_3^{T_{1,2,3}} \cdots a_n^{T_{1,2,n}} \cdot 
    \Big( \prod_{j=3}^{n-1} (a_j a_{j+1}^{T_{1,j,j+1}} \cdots 
                     a_n^{T_{1,j,n}})^{R_{1,2,j}(T;1,v)} \Big)
    \cdot a_n^{R_{1,2,n}(T; 1,v)}.
\end{eqnarray*}
Using the multiplication and powering polynomials for $U(T)$ and the 
conjugation polynomials $R_{1,2,j}$ with $j < n$ we determine 
polynomials $g_2, \ldots, g_n \in \Q[T][v]$ with
\[ a_2 a_3^{T_{1,2,3}} \cdots a_n^{T_{1,2,n}} \cdot 
    \Big( \prod_{j=3}^{n-1} (a_j a_{j+1}^{T_{1,j,j+1}} \cdots 
                     a_n^{T_{1,j,n}})^{R_{1,2,j}(T;1,v)} \Big)
  = a_2^{g_2(v)} \cdots a_n^{g_n(v)}.\]
Then it follows that $R_{1,2,n}(T;1,v+1) = R_{1,2,n}(T;1,v) + g_n(v)$.
\end{proof}

Lemma \ref{step2} and $R_{1,2,n}(T;1,0) = 0$ yield that $R_{1,2,n}(T;1,v)$ 
as a function in $v$ satisfies the assumptions of Lemma \ref{bernoulli} and 
this, in turn, allows to read off a polynomial for $R_{1,2,n}(T;1,v)$.

\subsection{Multiplication polynomials}
\label{mult}

We determine the multiplication polynomials $F_1, \ldots,F_n$.
By Section \ref{induct} it is sufficient to consider $F_n$. Let 
$x = (x_1, \ldots, x_n)$ and $y = (y_1, \ldots, y_n)$. Then 
\begin{align*}
  & a_1^{F_1(T;x,y)} \cdots a_n^{F_n(T;x,y)} \\
= \ & a_1^{x_1} \cdots a_n^{x_n} \cdot a_1^{y_1} \cdots a_n^{y_n} \\
= \ & a_1^{x_1+y_1} (a_2^{x_2})^{a_1^{y_1}} \cdots (a_n^{x_n})^{a_1^{y_1}}
     \cdot a_2^{y_2} \cdots a_n^{y_n} \\
= \ & a_1^{x_1+y_1} \Big( \prod_{i=2}^n 
    a_i^{x_i} a_{i+1}^{R_{1,i,i+1}(T;x_i,y_1)} \cdots 
    a_n^{R_{1,i,n}(T;x_i,y_1)} \Big)
    \cdot a_2^{y_2} \cdots a_n^{y_n}. \tag{$*$}
\end{align*}
Polynomials for $R_{i,j,k}$ can be computed as described in
Section \ref{conj}. Based on this, we can use the multiplication and 
powering in the subgroup $U(T)$ as defined in Section \ref{induct}
to evaluate the right-hand side of the product $(*)$. This yields 
a polynomial $F_n(T;x,y)$.

\begin{remark}
\label{formF}
As $a_n$ is central in $G(T)$, the polynomial $F_n(T;x,y)$ has the form
\[ F_n(T; x, y) 
   \ = \ x_n + y_n + H_n(T; x_1, \ldots, x_{n-1}, y_1, \ldots, y_{n-1}),\]
where $H_n$ is a polynomial in $x_1, \ldots, x_{n-1}, y_1, \ldots, y_{n-1}$ 
with parameters $T$.
\end{remark}

\subsection{Powering polynomials}
\label{power}

We determine the powering polynomials $K_1, \ldots, K_n$. By 
Section \ref{induct} it is sufficient to consider $K_n$. The next
lemma shows that $K_n$ satisfies a polynomial recursion. Its 
constructive proof translates to an algorithm to determine such 
a recursion for $K_n$.

\begin{lemma}
\label{step3}
There exists $h(z) \in \Q[T,x][z]$ with
$K_n(T;x,z+1) = K_n(T; x, z) + h(z)$ for all $z \in \Z$.
\end{lemma}

\begin{proof}
We evaluate
\begin{eqnarray*}
&& (a_1^{x_1} \cdots a_n^{x_n})^{z+1} \\
&=& (a_1^{x_1} \cdots a_n^{x_n})^{z} 
      \cdot (a_1^{x_1} \cdots a_n^{x_n}) \\
&=& a_1^{K_1(T;x,z)} \cdots a_n^{K_n(T;x,z)}
      \cdot (a_1^{x_1} \cdots a_n^{x_n}) \\
&=& a_1^{F_1(T;K(T;x,z),x)} \cdots 
       a_n^{F_n(T;K(T;x,z),x)},
\end{eqnarray*}
where $K(T;x,z) = (K_1(T;x,z), \ldots, K_n(T;x,z))$. By Remark \ref{formF}, it follows that
\[ K_n(T; x, z+1) = F_n(T; K(T;x,z),x) = K_n(T;x, z) + x_n + h(T,x,z), \]
where $h$ is a polynomial in $T,x$ and $z$.
We consider $T$ and $x$ as parameters and $z$ as indeterminate to 
obtain the desired result.
\end{proof}

Lemma \ref{step3} and $K_n(T;x,0) = 0$ yield that $K_n(T;x,z)$ satisfies
the assumptions of Lemma~\ref{bernoulli}, where we consider $T$ and $x$
as parameters, and this, in turn, allows to determine the polynomial $K_n$.

\section{Consistency}
\label{consistent}

Let $n \in \N$ and recall that $\C_n = \{ t \in \Z^{\binom{n}{3}} \mid 
G(t) \mbox{ consistent}\}$. Our aim in this section is to investigate 
$\C_n$.
As before, let $T$ be a sequence of $\binom{n}{3}$ indeterminates and 
let \linebreak
$x = (x_1, \ldots, x_n)$, $y = (y_1, \ldots, y_n)$ and $w = (w_1, \ldots, 
w_n)$ be further indeterminates. Write \linebreak 
$F(T;x,y) = (F_1(T;x,y), \ldots, F_n(T;x,y))$ and define
\begin{eqnarray*} P(T; x,y,w) &=& F(T; F(T;x,y),w) - F(T;x,F(T;y,w)) \\
  &=& (P_1(T;x,y,w),
   \ldots, P_n(T;x,y,w)).
\end{eqnarray*}
Then $P_1, \ldots, P_n$ are polynomials in the indeterminates $x,y,w$
and the parameters $T$. Let $\{ C_1(T), \ldots, C_r(T) \} \subseteq 
\Q[T]$ denote the set of all coefficients of $P_1, \ldots, P_n$ 
considered as polynomials in $x,y,w$.

\begin{theorem}
If $t \in \C_n$, then $C_1(t) = \ldots = C_r(t) = 0$.
\end{theorem}

\begin{proof}
If $t \in \C_n$, then $G(t)$ is consistent. Recall that $a_1, \ldots, a_n$ 
are the generators of $G(t)$ and use the notation $a^x = a_1^{x_1} \cdots 
a_n^{x_n}$ for $x = (x_1, \ldots, x_n) \in \Z^n$. Let $x,y,w
\in \Z^n$. Then $(a^x a^y) a^w$ has a unique normal form $a^u$ in $G(t)$. 
Associativity asserts that $a^u = a^x (a^y a^w)$ and hence 
$u = F(t; F(t; x,y),w) = F(t; x, F(t; y,w))$. Thus $P_i(t; x,y,w) = 0$
for all $x,y,w \in \Z^n$ and $1 \leq i \leq n$. This implies that all
coefficients of all $P_i(t;x,y,w)$ are zero and thus $C_1(t) = \ldots = 
C_r(t) = 0$.
\end{proof}

Define $I_n(T)$ as the ideal in $\Q[T]$ generated by $C_1(T), \ldots, C_r(T)$.
Let $\hat{F}_n(T;x,y)$ be the remainder of $F_n(T;x,y)$ divided by $I_n(T)$ 
using a Groebner basis of $I_n(T)$ and, similarly, $\hat{K}_n(T;x,y)$ the 
remainder of $K_n(T;x,y)$. 

\begin{corollary}
If $t \in \C_n$, then $\hat{F}_n(t;x,y) = F_n(t;x,y)$ and 
$\hat{K}_n(t;x,z) = K_n(t;x,z)$ for each $x,y \in \Z^n$ and $z \in \Z$.
\end{corollary}

Experimental evidence suggest that the following conjecture holds. If
it holds, then it provides an interesting alternative description for
the consistency of presentations of the form $G(t)$.

\begin{conjecture}
If $t \in \{t \in \Z^{n \choose 3} \mid C_1(t) = \ldots = C_r(t)=0 \}$,
then $t \in \C_n$.
\end{conjecture}

\section{Implementation and results}
\label{impl}

We have implemented the algorithm resulting from Section \ref{hall} in the
HallPoly \cite{hallpoly} package of the computer algebra system {\sf GAP} 
\cite{GAP} using the {\sf GAP} interface \cite{singint} to {\sc Singular} 
\cite{DGPS} to facilitate Groebner basis calculations.

We used this implementation to determine the multiplication and powering 
polynomials $\hat{F}_n$ and $\hat{K}_n$ for $1 \leq n \leq 7$. For this
purpose we first determined $F_n$ and $K_n$ as described in Section
\ref{hall}, then we determined a (partial) Groebner basis for $I_n(T)$ 
and, based on this, obtained $\hat{F}_n$ and $\hat{K}_n$. 

Note that $I_n(T) = \{0\}$ if $n \leq 4$. A Groebner basis for $I_5(T)$ 
has $2$ elements and a Groebner basis for $I_6(T)$ has $21$ elements.
We were not able to compute a Groebner basis for $I_7(T)$, but we
determined a Groebner basis with degree-bound $7$ and this has $839$
elements. It is sufficient to reduce $F_7$ and $K_7$ to $\hat{F}_7$ and 
$\hat{K}_7$. 

The polynomials for $I_n(T)$, $\hat{F}_n$ and $\hat{K}_n$ for $n \leq 7$ 
are available in the HallPoly package. The following table contains a
summary on them.  We consider the polynomials as polynomials in $x,y,z$ 
with parameters $T$. The degree of such a polynomial is the maximal 
degree of a monomial in $x,y,z$ and the degree of a monomial is the sum 
of the exponents of its indeterminates.

\begin{center}
\begin{tabular}{|c || c | c || c | c|}
\hline
& \multicolumn{2}{ c ||}{multiplication $\hat{F}_n(T;x,y)$} 
& \multicolumn{2}{ c |}{powering $\hat{K}_n(T;x,z)$}\\ 
\hline
$n$ & degree & \# monomials & degree & \# monomials \\
\hline
1 & 1 & 2 & 2 & 1\\
2 & 1 & 2 & 2 & 1\\
3 & 2 & 3 & 4 & 3\\
4 & 3 & 8 & 6 & 13\\
5 & 4 & 26 & 8 & 43\\
6 & 5 & 83 & 10 & 127\\
7 & 6 & 266 & 12 & 354\\
\hline
\end{tabular}
\end{center}

This table exhibits that the number of monomials in both polynomials
$\hat{F}_n$ and $\hat{K}_n$ grows significantly with $n$. This is the
main reason why we could not complete the computation of $\hat{F}_8$ 
and $\hat{K}_8$: the polynomials became too large to be processed. This 
table also induces the following conjecture.

\begin{conjecture} \label{degconj}
$\hat{F}_n(T;x,y)$ has degree  $n-1$ and $\hat{K}_n(T;x,z)$ has degree 
$2(n-1)$.
\end{conjecture}

\section{Multiplication methods in $T$-groups}
\label{comp}

Suppose that $G$ is a $T$-group of Hirsch length at most $7$ given by a 
consistent nilpotent presentation $G(t)$. The Hall polynomials determined
here facilitate an effective multiplication in $G$ via the following
approach. Write $a^{x}$ for $a_1^{x_1} \cdots a_n^{x_n}$:
\begin{items}
\bulit
Precomputation: Determine $\hat{F}_1(t;x,y), \ldots, \hat{F}_n(t;x,y)$ 
by evaluating $T$ to $t$ in these polynomials.
\bulit
Multiplication: Calculate $a^{\ol{x}} a^{\ol{y}}$ by evaluating $x$ to
$\ol{x}$ and $y$ to $\ol{y}$ in the polynomials of the first step.
\end{items}
\medskip

There are various other methods known to perform the multiplication in a 
$T$-group $G$. We discuss three of them in the following.
\begin{items}
\item[\rm (1)]
The 'collection from the left' algorithm, see \cite{LGS90}, performs an
iterated application of the relations of the group until the result is in
normal form. An implementation is available in the Polycyclic package 
\cite{polycyc}.
\item[\rm (2)]
The computation of Hall polynomials via 'Deep Thought', see \cite{LGS98},
first computes the Hall polynomials from the considered presentation and
then uses the evaluation of polynomials. An implementation is available
in the Deep Thought package \cite{deepth}.
\item[\rm (3)]
The 'Malcev collection', see \cite{AL07}, computes the Lie algebra 
corresponding to $G$ via the Malcev correspondence and then uses this
to translate the multiplication to addition in the Lie algebra. 
An implementation is available in the Guarana package \cite{guarana}.
\end{items}

Our method as well as the methods (2) and (3) are similar in the respect
that they first need a precomputation step and then, based on this, have 
a highly effective multiplication method. The three precomputation steps 
are quite different. Our precomputation step requires the evaluation of 
given polynomials and thus is highly effective, but, on the other hand, 
this method is limited to Hirsch length at most $7$. The precomputation 
of (2) requires an application of the Deep Thought algorithm and the 
precomputation step of (3) requires the computation of the setup of the 
Malcev correspondence.

Method (1) is of an entirely different nature. Its runtime depends heavily
on the size of the input. Nonetheless it has the advantage that it does
not require a precomputation step and thus can be used much more readily.


\end{document}